\documentclass[12pt]{amsart}
\usepackage{graphicx}
\usepackage{amssymb}
\newtheorem{thm}{Theorem}[section]
\newtheorem{cor}[thm]{Corollary}

\newtheorem{prop}[thm]{Proposition}
\theoremstyle{definition}
\newtheorem{defn}[thm]{Definition}
\theoremstyle{remark}

\numberwithin{equation}{section}

\newcommand{\h}{\mathcal{H}}

\newcommand{\K}{\mathcal{K}}

\begin{document}
\title[Invertibility of Multipliers]{Invertibility of Multipliers for Continuous G-frames}%
\author {Y. Khedmati and M. R. Abdollahpour }%
\address{
\newline
\indent Department of Mathematics
\newline
\indent Faculty of Sciences
\newline \indent   University of Mohaghegh Ardabili
\newline \indent  Ardabil 56199-11367
\newline \indent Iran}
\email{m.abdollah@uma.ac.ir, mrabdollahpour@yahoo.com}
\email{khedmati.y@uma.ac.ir, khedmatiy.y@gmail.com}
%\address{Department of mathematics ,Tabriz university, Tabriz, Iran}%
%\email{$\text{mr}_{-}$abdollahpour@yahoo.com}
%\email{mhfaroughi@yahoo.com}
%\thanks{}%

%\date{}%
%\dedicatory{}%
%\commby{}%
% ----------------------------------------------------------------
\begin{abstract}
In this paper we study the concept of multipliers for continuous $g$-Bessel families in Hilbert spaces. We present necessary conditions for invertibility of multipliers for continuous $g$-Bessel families and sufficient conditions for invertibility of multipliers for continuous $g$-frames.
\end{abstract}
\maketitle
% ----------------------------------------------------------------
\section{{\textbf Introduction}}
In 1952, the concept of frames for Hilbert spaces was defined by Duffin and Schaeffer \cite{DS}. Frames are important tools in the signal processing, image processing, data compression, etc. In 1993, Ali, Antoine and Gazeau developed the notion of ordinary frame to a family indexed by a measurable space which are known as continuous frames \cite{Ali2}.
 In 2006, $g$-frames or generalized frames introduced by Sun \cite{ws}. Abdollahpour and Faroughi introduced and investigated continuous $g$-frames and Riesz-type continuous $g$-frames \cite{ostad}.
 \par
In the rest of this paper $(\Omega,\mu)$ is a measure space with positive measure $\mu$, $\{\K_\omega:\omega\in \Omega\}$ is a family of Hilbert spaces and $GL(\h)$ denotes the set of all invertible bounded linear operators on Hilbert space $\h$.
\par
In 2007, the Bessel multiplier for Bessel sequences in Hilbert spaces was introduced by P. Balazs \cite{base}. 
\begin{defn}
Let $\h$ and $\K$ be Hilbert spaces. Suppose that $F=\{f_i\}_{i\in I}$ and $G=\{g_i\}_{i\in I}$ are Bessel sequences for $\h$ and $\K,$ respectively, and $m=\{m_i\}_{i\in I}\in l^{\infty}(I).$ The operator $M_{m,F,G}f:\h\rightarrow\K$ defined by
\begin{align*}
M_{m,F,G}f=\sum_{i\in I}m_i\langle f,f_i\rangle g_i.
\end{align*}
is called the Bessel multiplier for $F$ and $G.$
\end{defn}
In this paper we obtain necessary conditions for invertibility of multipliers for continuous $g$-Bessel families and sufficient conditions for invertibility of multipliers for continuous $g$-frames, by extending some results of \cite{base}. In the rest of this section, we summarize some facts about continuous $g$-frames and multipliers of continuous $g$-Bessel families from \cite{ostad}, \cite{alizad}.
%%%%%%%%%%%%%%%%%%%%%%%%%%%%%%%%%%%%%%%%%%%%%%%%%%%%%%%%%%%%%%%%%%%%%%%%%%%%
\par We say that $F\in\prod_{\omega\in \Omega}\K_{\omega}$ is
strongly measurable if $F$ as a mapping of $\Omega$ to
$\bigoplus_{\omega\in \Omega}\K_{\omega}$ is measurable, where
$$\prod_{\omega\in \Omega}\K_{\omega}=\left\{f:\Omega\rightarrow\bigcup_{\omega\in \Omega}\K_{\omega}
:f(\omega)\in \K_{\omega}\right\}.$$
%%%%%%%%%%%%%%%%%%%%%%%%%%%%%%%%%%%%%%%%%%%%%%%%%%%%%%%%%%%%%%%%%%%%%%%%%%%%
\begin{defn}
We say that $\Lambda=\{\Lambda_{\omega}\in
B(\h,\K_{\omega}):\,\omega\in\Omega\}$ is a continuous $g$-frame for $\h$  with respect to
$\{\K_{\omega}:\omega\in \Omega\}$ if
\begin{enumerate}
\item[(i)] for each $f\in \h$, $\{\Lambda_{\omega}f:\omega\in \Omega\}$ is
strongly measurable,
\item[(ii)]
there are two constants $0<A_\Lambda\leq
B_\Lambda<\infty$ such that
\begin{equation}\label{cgframe}
A_\Lambda\|f\|^{2}\leq\int_{\Omega}\|\Lambda_\omega f\|^{2}
d\mu(\omega)\leq B_\Lambda\|f\|^{2},\; f\in \h.
\end{equation}
\end{enumerate}
We call $A_\Lambda,B_\Lambda$ the lower and upper continuous $g$-frame bounds,
respectively. $\Lambda$ is called a tight continuous $g$-frame if $A_\Lambda=B_\Lambda,$ and a
Parseval continuous $g$-frame if $A_\Lambda=B_\Lambda=1.$ 
%If for each $\omega\in\Omega,$
%$\K_{\omega}=\K,$ then $\Lambda$ is called a continuous $g$-frame for $\h$
%with respect to $\K$. 
A family $\Lambda=\{\Lambda_{\omega}\in
B(\h,\K_{\omega}):\,\omega\in\Omega\}$ is called a continuous
$g$-Bessel family for $\h$  with respect to $\{\K_{\omega}:\omega\in \Omega\}$ if the right hand side in the inequality (\ref{cgframe})
holds for all $f\in \h,$ in this case, $B_\Lambda$ is called the continuous $g$-Bessel constant.
\end{defn}
%%%%%%%%%%%%%%%%%%%%%%%%%%%%%%%%%%%%%%%%%%%%%%%%%%%%%%%%%%%%%%%%%%%%%%%%%%%%
\begin{prop}\label{soos}\cite{ostad}
Let $\Lambda=\{\Lambda_{\omega}\in B(\h,\K_{\omega}):\omega \in \Omega\}$ be a continuous
$g$-frame. There exists a unique positive and
invertible operator $S_\Lambda:\h\rightarrow \h$ such that
\begin{eqnarray*}\label{di}
\langle S_\Lambda f,g\rangle =\int_{\Omega}\langle \Lambda_\omega f,
\Lambda_{\omega}g\rangle d\mu(\omega),\quad f,g\in \h,
\end{eqnarray*}
 and $A_\Lambda I\leq S_\Lambda\leq B_\Lambda I.$
\end{prop}
%%%%%%%%%%%%%%%%%%%%%%%%%%%%%%%%%%%%%%%%%%%%%%%%%%%%%%%%%%%%%%%%%%%%%%%%%%%%
The operator $S_\Lambda$ in Proposition \ref{soos} is called the continuous
$g$-frame operator of $\Lambda.$ 
%Also, we have
%\begin{eqnarray*}\label{abcd}
%\langle f,g \rangle &=&\int_\Omega \langle \Lambda_\omega S_\Lambda^{-1}f,\Lambda_\omega g \rangle \,d\mu(\omega)\\&=&\int_\Omega \langle \Lambda_\omega f,\Lambda_\omega S_\Lambda^{-1} g
%\rangle\, d\mu(\omega),
%\end{eqnarray*}
% for all $f,g\in \h$.
 %%%%%%%%%%%%%%%%%%%%%%%%%%%%%%%%%%%%%%%%%%%%%%%%%%%%%%%%%%
\par We consider the space
$$\widehat {\K}=\left\{F\in \prod_{\omega\in\Omega}\K_{\omega}:
\text{F is strongly measurable},\:
\int_{\Omega}\|F(\omega)\|^{2}d\mu(\omega)<\infty\right\}.$$
It is clear that $\widehat {\K}$ is a Hilbert space with point wise operations and with the inner
product given by
$$\langle F,G \rangle=\int_{\Omega}\langle
F(\omega),G(\omega)\rangle d\mu(\omega).$$
%%%%%%%%%%%%%%%%%%%%%%%%%%%%%%%%%%%%%%%%%%%%%%%%%%%%%%%%%%%%%%%%%%%%%%%%%%%%
\begin{prop}\cite{ostad}\label{combi}
Let $\Lambda=\{\Lambda_{\omega}\in B(\h,\K_{\omega}):\omega \in \Omega\}$ be a continuous
$g$-Bessel family. Then the mapping $T_\Lambda:\widehat {\K}\rightarrow\h$
defined by
\begin{eqnarray}\label{ti}
\langle T_\Lambda F,g\rangle=\int_{\Omega}\langle\Lambda_{\omega}^{*}F(\omega),g\rangle
d\mu(\omega),\;F\in\widehat {\K},\,g\in\h
\end{eqnarray}
is linear and bounded with $\|T_{\Lambda}\|\leq\sqrt{B_\Lambda}.$ Also, for
each $g\in \h$ and $\omega\in\Omega,$ we have
$$(T_\Lambda^{*}g)(\omega)=\Lambda_{\omega}g.$$
\end{prop}
%%%%%%%%%%%%%%%%%%%%%%%%%%%%%%%%%%%%%%%%%%%%%%%%%%%%%%%%%%%%%%%%%%%%%%%%%%%%
%\begin{thm}\cite{ostad}\label{ctf}
%Let $(\Omega,\mu)$ be a measure space, where $\mu$ is
%$\sigma$-finite. Suppose that $\Lambda=\{\Lambda_{\omega}\in
%B(\h,\K_{\omega}):\omega \in \Omega\}$ is a family of operators such
% $\{\Lambda_{\omega}f:\omega\in\Omega\}$ is strongly measurable, for each $f\in \h.$
 %Then $\Lambda$ is a continuous $g$-frame if and only if the operator
%$T_\Lambda:\widehat {\K}\rightarrow \h$
%defined by (\ref{ti}) is bounded and onto.
%\end{thm}
%%%%%%%%%%%%%%%%%%%%%%%%%%%%%%%%%%%%%%%%%%%%%%%%%%%%%%%%%%%%%%%%%%%%%%%%%%%%
The operators $T_\Lambda$ and $T_\Lambda^{*}$ in Proposition \ref{combi} are called the synthesis and analysis operators of $\Lambda=\{\Lambda_{\omega}\in B(\h,\K_{\omega}):\omega \in \Omega\}$, respectively.
\begin{defn}
Let $\Lambda=\{\Lambda_{\omega}\in B(\h,\K_{\omega}):\omega \in \Omega\}$ and $\Theta=\{\Theta_{\omega}\in B(\h,\K_{\omega}):\omega \in \Omega\}$  be two continuous $g$-Bessel families such that
$$\langle f,g\rangle=\int_{\Omega}\langle \Theta_\omega f,\Lambda_{\omega}g\rangle d\mu(\omega),\; f,g\in\h,$$ then $\Theta$ is called a dual of $\Lambda$.
\end{defn}
%%%%%%%%%%%%%%%%%%%%%%%%%%%%%%%%%%%%%%%%%%%%%%%%%%%%%%%%%%
%By \cite[Proposition 3.2]{ostad}, if $\Lambda=\{\Lambda_{\omega}\in B(\h,\K_{\omega}):\omega \in \Omega\}$ is a dual of  $\Theta=\{\Theta_{\omega}\in B(\h,\K_{\omega}):\omega \in \Omega\},$ then $\Theta$ is a dual of $\Lambda.$
%%%%%%%%%%%%%%%%%%%%%%%%%%%%%%%%%%%%%%%%%%%%%%%%%%%%%%%%%%%%%%%%%%%%%%%%%%%%
Let $\Lambda=\{\Lambda_{\omega}\in B(\h,\K_{\omega}):\omega \in \Omega\}$ be a continuous $g$-frame. Then
$\widetilde\Lambda=\Lambda S_\Lambda^{-1}=\{\Lambda_\omega S_\Lambda^{-1}\in B(\h, \K_\omega):\omega \in \Omega\}$ is a continuous
$g$-frame and $\widetilde\Lambda$ is a dual of $\Lambda.$ We call
$\widetilde\Lambda$ the canonical dual of $\Lambda$. In this paper, we will show by $\widetilde\Lambda$ the canonical dual of $\Lambda.$
%%%%%%%%%%%%%%%%%%%%%%%%%%%%%%%%%%%%%%%%%%%%%%%%%%%%%%%%%%%%%%%%%%%%%%%%%%%%
%\begin{defn}
%Let $\Lambda=\{\Lambda_{\omega}\in B(\h, \K_{\omega}):\omega \in \Omega\}$ and
%$\Theta=\{\Theta_{\omega}\in B(\K, \K_{\omega}):\omega \in \Omega\}$ be two continuous $g$-frames.
%Then $\Lambda$ and $\Theta$ are said similar if there is an invertible
%operator $S:\h\rightarrow \K$ such that
%$\Theta_{\omega}S=\Lambda_{\omega}$ for a.e. $\omega\in\Omega.$
%\end{defn}
%%%%%%%%%%%%%%%%%%%%%%%%%%%%%%%%%%%%%%%%%%%%%%%%%%%%%%%%%%%%%%%%%%%%%%%%%%%%
\par Two continuous $g$-Bessel families $\Lambda=\{\Lambda_{\omega}\in B(\h, \K_{\omega}):\omega \in \Omega\}$ and $\Theta=\{\Theta_{\omega}\in B(\h, \K_{\omega}):\omega \in \Omega\}$ are weakly equal, if for all $f\in \h$,
\begin{eqnarray*}
\Lambda_\omega f=\Theta_\omega f, \; a.e.\;\; \omega\in\Omega.
\end{eqnarray*}
%\par If the continuous $g$-frame
%$\Lambda=\{\Lambda_{\omega}\in B(\h, \K_{\omega}):\omega \in \Omega\}$ have only one
%dual(weakly), i.e., every dual of $\Lambda$ is weakly equal to the
%canonical dual of $\Lambda$, then $\Lambda$ is called a Riesz-type continuous
%$g$-frame.
%%%%%%%%%%%%%%%%%%%%%%%%%%%%%%%%%%%%%%%%%%%%%%%%%%%%%%%
\begin{defn}\cite{yavgdual}
Let $\Lambda=\{\Lambda_{\omega}\in B(\h,\K_{\omega}):\omega \in \Omega\}$ and $\Theta=\{\Theta_{\omega}\in B(\h,\K_{\omega}):\omega \in \Omega\}$  be  two continuous $g$-Bessel families. The family $\Theta$ is called a generalized dual of $\Lambda$ (or a $g$-dual of $\Lambda$), whenever the well defined operator $S_{\Lambda\Theta}:\h\rightarrow\h$,
\begin{eqnarray*}
\langle S_{\Lambda\Theta}f,g \rangle=\int_\Omega \langle \Theta_\omega f, \Lambda_\omega g \rangle d\mu(\omega),\quad  f,g\in\h.
\end{eqnarray*}
 is invertible.
\end{defn}
In the case that, the continuous $g$-Bessel family $\Theta$ is a $g$-dual of the continuous $g$-Bessel family $\Lambda,$  then $\Theta$ is a dual of a continuous $g$-Bessel family $\Lambda S_{\Theta\Lambda}^{-1}=\{\Lambda_\omega S_{\Theta\Lambda}^{-1}\in B
(\h,\K_\omega):\omega \in \Omega\},$ i.e.
\begin{eqnarray}\label{g-dual}
\langle f,g\rangle=\int_{\Omega}\langle \Theta_{\omega}f,
\Lambda_{\omega}S_{\Theta\Lambda}^{-1}g\rangle d\mu(\omega),\quad f,g\in\h.
\end{eqnarray}
%%%%%%%%%%%%%%%%%%%%%%%%%%%%%%%%%%%%%%%%%%%%%%%%%%%%%%%
\begin{prop}\label{s}\cite{alizad}
Let $\Lambda=\{\Lambda_{\omega}\in B(\h,\K_{\omega}):\omega \in \Omega\}$ and $\Theta=\{\Theta_\omega \in B(\h,\K_\omega):\omega \in \Omega \}$ be continuous
$g$-Bessel families and $m\in L^{\infty}(\Omega,\mu).$ The operator $M_{m,\Lambda,\Theta}:\h\rightarrow \h$ defined by 
\begin{eqnarray*}\label{di}
\langle M_{m,\Lambda,\Theta} f,g\rangle =\int_{\Omega}m(\omega)\langle \Theta_\omega f,
\Lambda_{\omega}g\rangle d\mu(\omega),\quad f,g\in \h,
\end{eqnarray*}
is a bounded operator with bound $\|m\|_{\infty}\sqrt{B_\Lambda B_\Theta}.$
\end{prop}
The operator $M_{m,\Lambda,\Theta}$ in Proposition \ref{s} is called the continuous $g$-Bessel multiplier for $\Lambda$ and $\Theta$ with respect to $m.$ Note that $M_{1,\Lambda,\Theta}=S_{\Lambda\Theta}.$
%%%%%%%%%%%%%%%%%%%%%%%%%%%%%%%%%%%%%%%%%%%%%%%%%%%%%%%
\begin{prop}\label{ss}\cite{alizad}
Let $\Lambda=\{\Lambda_{\omega}\in B(\h,\K_{\omega}):\omega \in \Omega\}$ and $\Theta=\{\Theta_\omega \in B(\h,\K_\omega):\omega \in \Omega \}$ be continuous
$g$-Bessel families and $m\in L^{\infty}(\Omega,\mu).$ Then 
\begin{eqnarray*}
M_{m,\Lambda,\Theta}^*= M_{\overline{m},\Theta,\Lambda}.
\end{eqnarray*}
\end{prop}
%%%%%%%%%%%%%%%%%%%%%%%%%%%%%%%%%%%%%%%%%%%%%%%%%%%%%%%
%%%%%%%%%%%%%%%%%%%%%%%%%%%%%%%%%%%%%%%%%%%%%%%%%%%%%%%
\section{{\textbf Invertibility of multipliers for continuous $g$-Bessel families}}
In this section we are going to get some results relevant to invertibility of continuous $g$-Bessel multipliers by generalizing results of \cite{dstova}. 
\par For any $\Lambda_\omega\in B(\h,\K_\omega), \omega\in\Omega$ and $m\in L^{\infty}(\Omega,\mu)$ we have 
\begin{align*}
\|(m(\omega)\Lambda_\omega) f\|
=\|m(\omega)\Lambda_\omega f\|
&=|m(\omega)|\|\Lambda_\omega f\|
\\&\leq\|m\|_\infty\|\Lambda_\omega\|\|f\|,\quad f\in\h,
\end{align*}
and therefore $m(\omega)\Lambda_\omega\in B(\h,\K_\omega).$

\begin{prop}\label{waw}
Let $\Lambda=\{\Lambda_{\omega}\in B(\h,\K_{\omega}):\omega \in \Omega\}$ be a continuous $g$-Bessel family and $m\in L^{\infty}(\Omega,\mu).$ Then
\begin{enumerate}
\item[(i)] $m\Lambda=\{m(\omega)\Lambda_{\omega}\in B(\h,\K_{\omega}):\omega \in \Omega\}$ is a continuous $g$-Bessel family with the continuous $g$-Bessel constant $B_\Lambda\|m\|_{\infty}^2.$
\item[(ii)] $M_{m,\Lambda,\Theta}= M_{1,\overline{m}\Lambda,\Theta}=M_{1,\Lambda,m\Theta},$ where $\overline{m}\Lambda=\{\overline{m(\omega)}\Lambda_{\omega}\in B(\h,\K_{\omega}):\omega \in \Omega\}$ and $m\Theta=\{m(\omega)\Theta_{\omega}\in B(\h,\K_{\omega}):\omega \in \Omega\}.$
\end{enumerate}
\end{prop}

\begin{proof}
(i) For any $f\in\h$ we have 
\begin{align*}
\int_{\Omega}\|m(\omega)\Lambda_\omega f\|^2d\mu(\omega)
&=\int_{\Omega}|m(\omega)|^2\|\Lambda_\omega f\|^2d\mu(\omega)
\\&\leq\|m\|_{\infty}^2\int_{\Omega}\|\Lambda_\omega f\|^2d\mu(\omega)
\\&\leq B_\Lambda\|m\|_{\infty}^2\|f\|^2.
\end{align*}
\\(ii) By (i), $\overline{m}\Lambda$ and $m\Theta$ are continuous $g$-Bessel families. For any $f,g\in\h$  we have 
\begin{align*}
\langle M_{m,\Lambda,\Theta} f,g\rangle 
=\int_{\Omega}m(\omega)\langle \Theta_\omega f,\Lambda_{\omega}g\rangle d\mu(\omega)
&=\int_{\Omega}\langle \Theta_\omega f,\overline{m(\omega)}\Lambda_{\omega}g\rangle d\mu(\omega)
\\&=\int_{\Omega}\langle m(\omega)\Theta_\omega f,\Lambda_{\omega}g\rangle d\mu(\omega).
\end{align*}
Therefore $M_{m,\Lambda,\Theta}=M_{1,\overline{m}\Lambda,\Theta}= M_{1,\Lambda,m\Theta}.$
\end{proof}

The following proposition gives necessary conditions for invertibility of multipliers for continuous $g$-Bessel families.
\begin{prop}\label{cong}
Let $\Lambda=\{\Lambda_{\omega}\in B(\h,\K_{\omega}):\omega \in \Omega\}$ and $\Theta=\{\Theta_\omega \in B(\h,\K_\omega):\omega \in \Omega \}$ be continuous
$g$-Bessel families and $0\neq m\in L^{\infty}(\Omega,\mu).$ If $M_{m,\Lambda,\Theta}\in GL(\h),$ then
\begin{enumerate}
\item[(i)] $m\Theta=\{m(\omega)\Theta_{\omega}\in B(\h,\K_{\omega}):\omega \in \Omega\}$ and $\overline{m}\Lambda=\{\overline{m(\omega)}\Lambda_{\omega}\in B(\h,\K_{\omega}):\omega \in \Omega\}$ are continuous $g$-frames with lower continuous $g$-frame bounds $(B_\Lambda\|M_{m,\Lambda,\Theta}^{-1}\|^2)^{-1}$ and $(B_\Theta\|M_{m,\Lambda,\Theta}^{-1}\|^2)^{-1},$ respectively.
\item[(ii)] $\Lambda$ and $\Theta$ are continuous $g$-frames with lower continuous $g$-frame bounds $(B_\Lambda\|m\|_\infty^2\|M_{m,\Lambda,\Theta}^{-1}\|^2)^{-1}$ and $(B_\Theta\|m\|_\infty^2\|M_{m,\Lambda,\Theta}^{-1}\|^2)^{-1},$ respectively.
\end{enumerate}
\end{prop}

\begin{proof}
(i) Since $M_{m,\Lambda,\Theta}\in GL(\h),$ by Proposition \ref{waw} (ii), the operators $M_{1,\Lambda,m\Theta}$ and $M_{1,\overline{m}\Lambda,\Theta}$ are invertible. Let $f\in\h$ and $f\neq 0,$ then by the proposition \ref{ss} we have
\begin{align*}
\|f\|^2=|\langle f,f\rangle|
&=|\langle f,M_{1,\Lambda,m\Theta}^{-1}M_{1,\Lambda,m\Theta}f\rangle|
\\&=|\langle M_{1,m\Theta,\Lambda}M_{1,m\Theta,\Lambda}^{-1}f,f\rangle|
\\&=\Big|\int_{\Omega}\langle \Lambda_\omega M_{1,m\Theta,\Lambda}^{-1}f,m(\omega)\Theta_{\omega}f\rangle d\mu(\omega)\Big|
\\&=|\langle T_\Lambda^* M_{1,m\Theta,\Lambda}^{-1}f, T_{m\Theta} ^*f\rangle|
\\&\leq\|T_\Lambda^* M_{1,m\Theta,\Lambda}^{-1}f\|\|T_{m\Theta} ^*f\|
\\&\leq\sqrt{B_\Lambda}\|M_{1,m\Theta,\Lambda}^{-1}\|\|f\|\|T_{m\Theta}^* f\|,
\end{align*}
and therefore we get 
\begin{align}\label{ooo}
\frac{1}{B_\Lambda\|M_{m,\Lambda,\Theta}^{-1}\|^2}\|f\|^2
=\frac{1}{B_\Lambda\|M_{1,m\Theta,\Lambda}^{-1}\|^2}\|f\|^2
&\leq\|T_{m\Theta}^* f\|^2
\nonumber\\&=\int_\Omega\|m(\omega)\Theta_\omega f\|^2 d\mu(\omega).
\end{align}
Also similarly we have 
\begin{eqnarray}\label{ttt}
\frac{1}{B_\Theta\|M_{m,\Lambda,\Theta}^{-1}\|^2}\|f\|^2
\leq\|T_{\overline{m}\Lambda}^* f\|^2=\int_\Omega\|\overline{m(\omega)}\Lambda_\omega f\|^2 d\mu(\omega).
\end{eqnarray}
It is clear that the inequlities (\ref{ooo}) and (\ref{ttt}) also hold for $f=0.$ 
 So by the Proposition \ref{waw} (i), $m\Theta$ and $\overline{m}\Lambda$ are continuous $g$-frames. 
\\(ii) For any $f\in\h$ by inequlity (\ref{ooo}) we have 
\begin{eqnarray*}
\frac{1}{B_\Lambda\|M_{m,\Lambda,\Theta}^{-1}\|^2}\|f\|^2\leq\int_\Omega\|m(\omega)\Theta_\omega f\|^2 d\mu(\omega)\leq\|m\|_{\infty}^2\int_\Omega\|\Theta_\omega f\|^2 d\mu(\omega),
\end{eqnarray*}
and therefore 
\begin{eqnarray*}
\frac{1}{B_\Lambda\|m\|_\infty^2\|M_{m,\Lambda,\Theta}^{-1}\|^2}\|f\|^2\leq\int_\Omega\|\Theta_\omega f\|^2 d\mu(\omega).
\end{eqnarray*}
Also similarly by inequlity (\ref{ttt}) we have 
\begin{eqnarray*}
\frac{1}{B_\Theta\|m\|_\infty^2\|M_{m,\Lambda,\Theta}^{-1}\|^2}\|f\|^2\leq\int_\Omega\|\Lambda_\omega f\|^2 d\mu(\omega).
\end{eqnarray*}
Thus $\Lambda$ and $\Theta$ are continuous $g$-frames.
\end{proof}
Note that the Proposition \ref{cong} (ii), generalizes the propsition 3.2 of \cite{ostad}. In the following proposition, by generalizing conclusion from \cite{basee} we get dual for continuous $g$-Bessel families $\Lambda$ and $\Theta$ when $M_{m,\Lambda,\Theta}\in GL(\h).$
%%%%%%%%%%%%%%%%%%%%%%%%%%%%%%%%%%%%%%%%%%%%%%%%
\begin{prop}
Let $\Lambda=\{\Lambda_{\omega}\in B(\h,\K_{\omega}):\omega \in \Omega\}$ and $\Theta=\{\Theta_\omega \in B(\h,\K_\omega):\omega \in \Omega \}$ be continuous
$g$-Bessel families and $0\neq m\in L^{\infty}(\Omega,\mu).$ If $M_{m,\Lambda,\Theta}\in GL(\h),$ then $\Theta$  and  $\overline{m}\Lambda M_{\overline{m},\Theta,\Lambda}^{-1}=\{\overline{m(\omega)}\Lambda_{\omega}M_{\overline{m},\Theta,\Lambda}^{-1}\in B(\h,\K_{\omega}):\omega \in \Omega\}$ are dual. Also $\Lambda$ and $m\Theta M_{m,\Lambda,\Theta}^{-1}=\{m(\omega)\Theta_{\omega}M_{m,\Lambda,\Theta}^{-1}\in B(\h,\K_{\omega}):\omega \in \Omega\}$ are dual.
\end{prop}

\begin{proof}
By the Proposition \ref{cong} (ii), $\Lambda$ and $\Theta$ are continuous $g$-frames. Since $M_{m,\Lambda,\Theta}\in GL(\h),$ $M_{\overline{m},\Theta,\Lambda}=M_{m,\Lambda,\Theta}^*\in GL(\h),$ and then by Propositin \ref{cong} (i) and \cite[Proposition 3.3]{ostad}, $\overline{m}\Lambda M_{\overline{m},\Theta,\Lambda}^{-1}=\{\overline{m(\omega)}\Lambda_{\omega}M_{\overline{m},\Theta,\Lambda}^{-1}\in B(\h,\K_{\omega}):\omega \in \Omega\}$ is a continuous $g$-frame. For any $f,g\in\h$ we have
\begin{align*}
\int_\Omega\langle\Theta_\omega f,\overline{m(\omega)}\Lambda_{\omega}M_{\overline{m},\Theta,\Lambda}^{-1}g\rangle d\mu(\omega)
&=\int_\Omega m(\omega)\langle\Theta_\omega f,\Lambda_{\omega}M_{\overline{m},\Theta,\Lambda}^{-1}g\rangle d\mu(\omega)
\\&=\langle M_{m,\Lambda,\Theta}f,M_{\overline{m},\Theta,\Lambda}^{-1}g\rangle=\langle f,g\rangle.
\end{align*}
Also
\begin{align*}
\int_\Omega\langle m(\omega)\Theta_\omega M_{m,\Lambda,\Theta}^{-1}f,\Lambda_\omega g\rangle d\mu(\omega)
&=\int_\Omega m(\omega)\langle\Theta_\omega M_{m,\Lambda,\Theta}^{-1}f,\Lambda_\omega g\rangle d\mu(\omega) 
\\&=\langle M_{m,\Lambda,\Theta}M_{m,\Lambda,\Theta}^{-1}g\rangle=\langle f,g\rangle.
\end{align*}
\end{proof}
%%%%%%%%%%%%%%%%%%%%%%%%%%%%%%%%%%%%%%%%%%%%%%%%
The following results give sufficient conditions for invertibility of multipliers for continuous $g$-frames.
\begin{thm}\label{main}
Let $\Lambda=\{\Lambda_{\omega}\in B(\h,\K_{\omega}):\omega \in \Omega\}$ be a continuous $g$-frame and $\Theta=\{\Theta_\omega \in B(\h,\K_\omega):\omega \in \Omega \}$ be a family of operators such that for each $f\in\h,$ $\{\Theta_\omega f\}_{\omega\in\Omega}$ is strongly measurable and there exists $\nu\in[0,\frac{A_\Lambda^2}{B_\Lambda})$ such that 
\begin{eqnarray}\label{ffirs}
\int_\Omega\|(\Lambda_\omega-\Theta_\omega) f\|^2 d\mu(\omega)\leq\nu\|f\|^2,\quad f\in\h.
\end{eqnarray}
Suppose $m\in L^{\infty}(\Omega,\mu)$ such that for some positive constant $\delta$ we have $m(\omega)\geq\delta>0$ a.e.  and $\frac{\|m\|_\infty}{\delta}\sqrt{\nu}<\frac{A_\Lambda}{\sqrt{B_\Lambda}}.$
Then $M_{m,\Lambda,\Theta}\in GL(\h)$ and
\begin{eqnarray*}
\frac{1}{\|m\|_\infty B_\Lambda+\|m\|_\infty\sqrt{\nu B_\Lambda}}\|f\|\leq\|M_{m,\Lambda,\Theta}^{-1}f\|\leq\frac{1}{\delta A_\Lambda-\|m\|_\infty\sqrt{\nu B_\Lambda}}\|f\|,\quad f\in\h,
\end{eqnarray*}
and 
\begin{eqnarray*}
M_{m,\Lambda,\Theta}^{-1}=\sum_{k=0}^{\infty}\big[S_{\sqrt{m}\Lambda}^{-1}(S_{\sqrt{m}\Lambda}-M_{m,\Lambda,\Theta})\big]^k S_{\sqrt{m}\Lambda}^{-1},
\end{eqnarray*}
where 
\begin{align*}
&\big\|M_{m,\Lambda,\Theta}^{-1}-\sum_{k=0}^{n}\big[S_{\sqrt{m}\Lambda}^{-1}(S_{\sqrt{m}\Lambda}-M_{m,\Lambda,\Theta})\big]^k S_{\sqrt{m}\Lambda}^{-1}\big\|
\\&\leq\Big(\frac{\|m\|_\infty\sqrt{\nu B_\Lambda}}{\delta A_\Lambda}\Big)^{n+1}\frac{1}{\delta A_\Lambda-\|m\|_\infty\sqrt{\nu B_\Lambda}} ,\quad n\in\mathbb{N}.
\end{align*}
\end{thm}

\begin{proof}
If $\nu=0,$ then by inequlity (\ref{ffirs}), $\Lambda$ and $\Theta$ are weakly equal and so for any $f,g\in\h$
\begin{align*}
\langle M_{m,\Lambda,\Theta}f,g\rangle
=\int_\Omega m(\omega)\langle\Theta_\omega f,\Lambda_\omega g\rangle d\mu(\omega)
&=\int_\Omega m(\omega)\langle\Lambda_\omega f,\Lambda_\omega g\rangle d\mu(\omega)
\\&=\langle M_{m,\Lambda,\Lambda}f,g\rangle.
\end{align*}
Therefore by \cite[Proposition 3.3.]{alizad}, $M_{m,\Lambda,\Theta}=M_{m,\Lambda,\Lambda}=S_{\sqrt{m}\Lambda}$ is an invertible operator with lower and upper bounds $\delta A_\Lambda$ and $\|m\|_\infty B_\Lambda,$ respectively, where $\sqrt{m}\Lambda=\{\sqrt{m(\omega)}\Lambda_\omega\in B(\h,\K_{\omega}):\omega \in \Omega\}.$ Therefore for any $f\in\h$ we have
\begin{align}\label{star1}
\frac{1}{\|m\|_\infty B_\Lambda}\|f\|\leq\|M_{m,\Lambda,\Lambda}^{-1}f\|=\|S_{\sqrt{m}\Lambda}^{-1}f\|\leq\frac{1}{\delta A_\Lambda}\|f\|.
\end{align}
For $\nu>0,$ by inequlity (\ref{ffirs}), the family $\Lambda-\Theta=\{\Lambda_\omega-\Theta_\omega\in B(\h,\K_\omega):\omega\in\Omega\}$ is a continuous $g$-Bessel family and so $\Theta$ is a continuous $g$-Bessel family. Thus by the proposition \ref{s}, $M_{m,\Lambda,\Theta}$ is a well defined bounded operator. By (\ref{ffirs}), for any $f,g\in\h$ we have 
\begin{align*}
|\langle M_{m,\Lambda,\Theta}f-S_{\sqrt{m}\Lambda}f,g\rangle |
&=\Big|\int_\Omega m(\omega)\langle\Theta_\omega f,\Lambda_\omega g\rangle d\mu(\omega)-\int_\Omega m(\omega)\langle\Lambda_\omega f,\Lambda_\omega g\rangle d\mu(\omega)\Big|
\\&=\Big|\int_\Omega m(\omega)\langle(\Theta_\omega-\Lambda_\omega) f,\Lambda_\omega g\rangle d\mu(\omega)\Big|
\\&\leq\int_\Omega |m(\omega)|\big|\langle(\Theta_\omega-\Lambda_\omega) f,\Lambda_\omega g\rangle\big| d\mu(\omega)
\\&\leq\|m\|_\infty\int_\Omega||(\Theta_\omega-\Lambda_\omega) f|| ||\Lambda_\omega g|| d\mu(\omega)
\\&\leq\|m\|_\infty\Big(\int_\Omega||(\Theta_\omega-\Lambda_\omega) f||^2d\mu(\omega)\Big)^{\frac{1}{2}}\Big(\int_\Omega||\Lambda_\omega g||^2d\mu(\omega)\Big)^{\frac{1}{2}}
\\&\leq\|m\|_\infty\sqrt{\nu B_\Lambda}\|f\|\|g\|.
\end{align*}
Therefore we have
\begin{align}\label{star2}
\|M_{m,\Lambda,\Theta}f-S_{\sqrt{m}\Lambda}f\|\leq\|m\|_\infty\sqrt{\nu B_\Lambda}\|f\|.
\end{align}
Since $\|m\|_\infty\sqrt{\nu B_\Lambda}<\delta A_\Lambda\leq\frac{1}{\|S_{\sqrt{m}\Lambda}^{-1}\|},$ by \cite[Proposition 2.2.]{dstova}, $M_{m,\Lambda,\Theta}\in GL(\h)$ and 
\begin{eqnarray*}
M_{m,\Lambda,\Theta}^{-1}=\sum_{k=0}^{\infty}\big[S_{\sqrt{m}\Lambda}^{-1}(S_{\sqrt{m}\Lambda}-M_{m,\Lambda,\Theta})\big]^k S_{\sqrt{m}\Lambda}^{-1}.
\end{eqnarray*}
Also by inequality (\ref{star1}) for any $f\in\h$
\begin{align*}
\frac{1}{\|m\|_\infty\sqrt{\nu B_\Lambda}+\|m\|_\infty B_\Lambda}\|f\|
&\leq\frac{1}{\|m\|_\infty\sqrt{B_\Lambda \nu}+\|S_{\sqrt{m}\Lambda}\|}\|f\|
\\&\leq\|M_{m,\Lambda,\Theta}^{-1}f\|
\\&\leq\frac{1}{\frac{1}{\|S_{\sqrt{m}\Lambda}^{-1}\|}-\|m\|_\infty\sqrt{B_\Lambda \nu}}\|f\|
\\&\leq\frac{1}{\delta A_\Lambda-\|m\|_\infty\sqrt{\nu B_\Lambda}}\|f\|.
\end{align*}
Since $\frac{\|m\|_\infty}{\delta}\sqrt{\nu}<\frac{A_\Lambda}{\sqrt{B_\Lambda}},$                                                         $\frac{\|m\|_\infty \sqrt{\nu B_\Lambda}}{\delta A_\Lambda}<1.$ By inequalities (\ref{star1}) and (\ref{star2})
for $n\in\mathbb{N}$ we have
\begin{align*}
&\big\|M_{m,\Lambda,\Theta}^{-1}-\sum_{k=0}^{n}\big[S_{\sqrt{m}\Lambda}^{-1}(S_{\sqrt{m}\Lambda}-M_{m,\Lambda,\Theta})\big]^k S_{\sqrt{m}\Lambda}^{-1}\big\|
\\&=\big\|\sum_{k=n+1}^{\infty}\big[S_{\sqrt{m}\Lambda}^{-1}(S_{\sqrt{m}\Lambda}-M_{m,\Lambda,\Theta})\big]^k S_{\sqrt{m}\Lambda}^{-1}\big\|
\\&\leq\big\|S_{\sqrt{m}\Lambda}^{-1}\big\|\sum_{k=n+1}^{\infty}\big\|S_{\sqrt{m}\Lambda}^{-1}\big\|^k\big\|S_{\sqrt{m}\Lambda}-M_{m,\Lambda,\Theta} \big\|^k
\\&\leq\frac{1}{\delta A_\Lambda}\sum_{k=n+1}^{\infty}\Big(\frac{\|m\|_\infty\sqrt{\nu B_\Lambda}}{\delta A_\Lambda}\Big)^k
\\&=\Big(\frac{\|m\|_\infty\sqrt{\nu B_\Lambda}}{\delta A_\Lambda}\Big)^{n+1}\frac{1}{\delta A_\Lambda-\|m\|_\infty\sqrt{\nu B_\Lambda}}.
\end{align*}
\end{proof}
Note that the Theorem \ref{main} generalizes the Proposition 3.3 of \cite{alizad}.
%%%%%%%%%%%%%%%%%%%%%%%%%%%%%%%%%%%%%%%%%%%%%%%%%%
\begin{prop}\label{cooor}
Let $\Lambda=\{\Lambda_{\omega}\in B(\h,\K_{\omega}):\omega \in \Omega\}$ be a continuous $g$-frame. 
Let $m\in L^{\infty}(\Omega,\mu)$ such that $\|m-1\|_\infty\leq \lambda<\frac{A_\Lambda}{B_\Lambda}$ for some $\lambda.$
Then $M_{m,\Lambda,\Lambda}\in GL(\h)$ and
\begin{eqnarray*}
\frac{1}{(\lambda+1)B_\Lambda}\|f\|\leq\|M_{m,\Lambda,\Lambda}^{-1}f\|\leq\frac{1}{ A_\Lambda-\lambda B_\Lambda}\|f\|,\quad f\in\h,
\end{eqnarray*}
and 
\begin{eqnarray*}%\label{alpha}
M_{m,\Lambda,\Lambda}^{-1}=\sum_{k=0}^{\infty}\big[S_\Lambda^{-1}(S_\Lambda-M_{m,\Lambda,\Lambda})\big]^k S_\Lambda^{-1},
\end{eqnarray*}
where
\begin{align*}
&\big\|M_{m,\Lambda,\Lambda}^{-1}-\sum_{k=0}^{n}\big[S_\Lambda^{-1}(S_\Lambda-M_{m,\Lambda,\Lambda})\big]^k S_\Lambda^{-1}\big\|
\leq\big(\frac{\lambda B_\Lambda}{A_\Lambda}\big)^{n+1}\frac{1}{A_\Lambda-\lambda B_\Lambda},\quad n\in\mathbb{N}.
\end{align*}
\end{prop}

\begin{proof}
For any $f,g\in\h$ we have 
\begin{align*}
|\langle M_{1,\Lambda,m\Lambda}f-S_{\Lambda}f,g\rangle|
&=\Big|\int_\Omega\big\langle(m(\omega)-1)\Lambda_\omega f,\Lambda_\omega g\big\rangle d\mu(\omega)\Big|
\\&\leq\int_\Omega |m(\omega)-1|\big|\langle\Lambda_\omega f,\Lambda_\omega g\rangle\big| d\mu(\omega)
\\&\leq\|m-1\|_\infty\int_\Omega||\Lambda_\omega f|| ||\Lambda_\omega g|| d\mu(\omega)
\\&\leq\|m-1\|_\infty\Big(\int_\Omega||\Lambda_\omega f||^2d\mu(\omega)\Big)^{\frac{1}{2}}\Big(\int_\Omega||\Lambda_\omega g||^2d\mu(\omega)\Big)^{\frac{1}{2}}
\\&\leq\lambda B_\Lambda\|f\|\|g\|.
\end{align*}
Therefore we have
\begin{align*}
\|M_{1,\Lambda,m\Lambda}f-S_{\Lambda}f\|\leq\lambda B_\Lambda\|f\|.
\end{align*}
Since $0\leq\lambda B_\Lambda<A_\Lambda\leq\frac{1}{\|S_\Lambda^{-1}\|},$ similar to the proof of the Theorem \ref{main}, by \cite[Proposition 2.2.]{dstova} and the Proposition \ref{waw} (ii), the proof is completed.
\end{proof}
%%%%%%%%%%%%%%%%%%%%%%%%%%%%%%%%%%%%%%%%%%%%%%%%%%
\begin{thm}
Let $\Lambda=\{\Lambda_{\omega}\in B(\h,\K_{\omega}):\omega \in \Omega\}$ be a continuous $g$-frame and $\Theta=\{\Theta_\omega \in B(\h,\K_\omega):\omega \in \Omega \}$ be a family of operators such that for each $f\in\h,$ $\{\Theta_\omega f\}_{\omega\in\Omega}$ is strongly measurable. Suppose there exists $\nu\in[0,\frac{A_\Lambda^2}{B_\Lambda})$ such that the inequlity (\ref{ffirs}) is satisfied.
Let $m\in L^{\infty}(\Omega,\mu)$ that $\|m-1\|_\infty\leq \lambda<\frac{A_\Lambda-\sqrt{\nu B_\Lambda}}{B_\Lambda+\sqrt{\nu B_\Lambda}}$ for some $\lambda.$
Then $M_{m,\Lambda,\Theta}\in GL(\h)$ and
\begin{align*}
\frac{1}{(\lambda+1)(B_\Lambda+\sqrt{\nu B_\Lambda})}\|f\|\leq\|M_{m,\Lambda,\Theta}^{-1}f\|\leq\frac{1}{ A_\Lambda-\lambda B_\Lambda-(\lambda+1)\sqrt{\nu B_\Lambda}}\|f\|,\quad f\in\h,
\end{align*}
and 
\begin{align*}
M_{m,\Lambda,\Theta}^{-1}=\sum_{k=0}^{\infty}\big[M_{m,\Lambda,\Lambda}^{-1}(M_{m,\Lambda,\Lambda}-M_{m,\Lambda,\Theta})\big]^k M_{m,\Lambda,\Lambda}^{-1},
\end{align*}
where
\begin{align*}
&\big\|M_{m,\Lambda,\Theta}^{-1}-\sum_{k=0}^{n}\big[M_{m,\Lambda,\Lambda}^{-1}(M_{m,\Lambda,\Lambda}-M_{m,\Lambda,\Theta})\big]^k M_{m,\Lambda,\Lambda}^{-1}\big\|
\\&\leq\Big(\frac{(\lambda+1)\sqrt{\nu B_\Lambda}}{A_\Lambda-\lambda B_\Lambda}\Big)^{n+1}\frac{1}{A_\Lambda-\lambda B_\Lambda-(\lambda+1)\sqrt{\nu B_\Lambda}},\quad n\in\mathbb{N}.
\end{align*}
\end{thm}

\begin{proof}
If $\nu=0,$ by the inequlity (\ref{ffirs}), $\Lambda$ and $\Theta$ are weakly equal. Also for $\nu=0$ we have $\|m-1\|_\infty\leq \lambda<\frac{A_\Lambda}{B_\Lambda}.$ Then by the Proposition \ref{cooor}, for $\nu=0$ the proof is completed. For $\nu\neq0$ by the inequlity (\ref{ffirs}), the family $\Lambda-\Theta$ is a continuous $g$-Bessel family and so $\Theta$ is a continuous $g$-Bessel family. Similar to the proof of the Theorem \ref{main}, for any $f,g\in\h$ we have
\begin{align*}
|\langle M_{m,\Lambda,\Theta}f- M_{m,\Lambda,\Lambda}f,g\rangle |
&=\Big|\int_\Omega m(\omega)\langle(\Theta_\omega-\Lambda_\omega) f,\Lambda_\omega g\rangle d\mu(\omega)\Big|
\\&\leq\|m\|_\infty\Big(\int_\Omega||(\Theta_\omega-\Lambda_\omega) f||^2d\mu(\omega)\Big)^{\frac{1}{2}}\Big(\int_\Omega||\Lambda_\omega g||^2d\mu(\omega)\Big)^{\frac{1}{2}}
\\&\leq\|m\|_\infty\sqrt{\nu B_\Lambda}\|f\|\|g\|.
\end{align*}
Thus by $\|m-1\|_\infty\leq\lambda,$ we have
\begin{align*}
\|M_{m,\Lambda,\Theta}f-M_{m,\Lambda,\Lambda}f\|
\leq\|m\|_\infty\sqrt{\nu B_\Lambda}\|f\|
&\leq (\lambda+1)\sqrt{\nu B_\Lambda}\|f\|.
\end{align*}
By $\lambda<\frac{A_\Lambda-\sqrt{\nu B_\Lambda}}{B_\Lambda+\sqrt{\nu B_\Lambda}}$ we have $(\lambda+1)\sqrt{\nu B_\Lambda}<A_\Lambda-\lambda B_\Lambda$ and since
$\|m-1\|_\infty\leq\lambda<\frac{A_\Lambda-\sqrt{\nu B_\Lambda}}{B_\Lambda+\sqrt{\nu B_\Lambda}}<\frac{A_\Lambda}{B_\Lambda},$ by the Proposition \ref{cooor}, we have $(\lambda+1)\sqrt{\nu B_\Lambda}<A_\Lambda-\lambda B_\Lambda\leq\frac{1}{\|M_{m,\Lambda,\Lambda}^{-1}\|}$ and $\|M_{m,\Lambda,\Lambda}\|\leq (\lambda+1)B_\Lambda.$ Therefore by \cite[Proposition 2.2.]{dstova}, $M_{m,\Lambda,\Theta}\in GL(\h)$ and for any $f\in\h,$ we have
\begin{align*}
\frac{1}{(\lambda+1)(B_\Lambda+\sqrt{\nu B_\Lambda})}\|f\|
&=\frac{1}{(\lambda+1)\sqrt{\nu B_\Lambda}+(\lambda+1) B_\Lambda}\|f\|
\\&\leq\frac{1}{(\lambda+1)\sqrt{\nu B_\Lambda}+\|M_{m,\Lambda,\Lambda}\|}\|f\|
\\&\leq\|M_{m,\Lambda,\Theta}^{-1}f\|
\\&\leq\frac{1}{\frac{1}{\|M_{m,\Lambda,\Lambda}^{-1}\|}-(\lambda+1)\sqrt{\nu B_\Lambda}}\|f\|
\\&\leq\frac{1}{A_\Lambda-\lambda B_\Lambda-(\lambda+1)\sqrt{\nu B_\Lambda}}\|f\|.
\end{align*}
and 
\begin{eqnarray*}
M_{m,\Lambda,\Theta}^{-1}=\sum_{k=0}^{\infty}\big[M_{m,\Lambda,\Lambda}^{-1}(M_{m,\Lambda,\Lambda}-M_{m,\Lambda,\Theta})\big]^k M_{m,\Lambda,\Lambda}^{-1},
\end{eqnarray*}
where for $n\in\mathbb{N}$ we have
\begin{align*}
&\big\|M_{m,\Lambda,\Theta}^{-1}-\sum_{k=0}^{n}\big[M_{m,\Lambda,\Lambda}^{-1}(M_{m,\Lambda,\Lambda}-M_{m,\Lambda,\Theta})\big]^k M_{m,\Lambda,\Lambda}^{-1}\big\|
\\&=\big\|\sum_{k=n+1}^{\infty}\big[M_{m,\Lambda,\Lambda}^{-1}(M_{m,\Lambda,\Lambda}-M_{m,\Lambda,\Theta})\big]^k M_{m,\Lambda,\Lambda}^{-1}\big\|
\\&\leq\big\|M_{m,\Lambda,\Lambda}^{-1}\big\|\sum_{k=n+1}^{\infty}\big\|M_{m,\Lambda,\Lambda}^{-1}\big\|^k\big\|M_{m,\Lambda,\Lambda}-M_{m,\Lambda,\Theta}\big\|^k 
\\&\leq\frac{1}{A_\Lambda-\lambda B_\Lambda}\sum_{k=n+1}^{\infty}\Big(\frac{(\lambda+1)\sqrt{\nu B_\Lambda}}{A_\Lambda-\lambda B_\Lambda}\Big)^{k}
\\&=\Big(\frac{(\lambda+1)\sqrt{\nu B_\Lambda}}{A_\Lambda-\lambda B_\Lambda}\Big)^{n+1}\frac{1}{A_\Lambda-\lambda B_\Lambda-(\lambda+1)\sqrt{\nu B_\Lambda}}.
\end{align*}
\end{proof}
%%%%%%%%%%%%%%%%%%%%%%%%%%%%%%%%%%%%%%%%%%%%%%%%%%%
\begin{prop}\label{newn}
Let $\Lambda=\{\Lambda_{\omega}\in B(\h,\K_{\omega}):\omega \in \Omega\}$ be a continuous $g$-frame and $S\in GL(\h).$ Suppose $m\in L^{\infty}(\Omega,\mu)$ is satisfied in one of the following conditions:                       
\begin{enumerate}
\item[(i)] for some positive constant $\delta,$ $m(\omega)\geq\delta>0$ a.e.
\item[(ii)] $\|m-1\|_\infty\leq\lambda<\frac{A_\Lambda}{B_\Lambda}$ for some $\lambda.$
\end{enumerate}
Then the operators $M_{m,\Lambda,\Lambda S}$ and $M_{m,\Lambda S,\Lambda}$ are invertible and 
\begin{eqnarray*}
M_{m,\Lambda,\Lambda S}^{-1}=S^{-1}M_{m,\Lambda,\Lambda}^{-1},\quad\quad M_{m,\Lambda S,\Lambda}^{-1}=M_{m,\Lambda,\Lambda}^{-1}(S^{-1})^*,
\end{eqnarray*}
where $\Lambda S=\{\Lambda_\omega S\in B(\h,K_\omega):\omega\in\Omega\}.$
\end{prop}

\begin{proof}
By \cite[Proposition 3.3]{ostad}, $\Lambda S$ is a continuous $g$-frame. For any $f,g\in\h$ we have 
\begin{align*}
\langle M_{m,\Lambda,\Lambda S}f,g\rangle=\int_\Omega m(\omega)\langle\Lambda_\omega Sf,\Lambda_\omega g\rangle d\mu(\omega)&=\langle M_{m,\Lambda,\Lambda}Sf,g\rangle,
\\\langle M_{m,\Lambda S,\Lambda}f,g\rangle=\int_\Omega m(\omega)\langle\Lambda_\omega f,\Lambda_\omega Sg\rangle d\mu(\omega)&=\langle M_{m,\Lambda,\Lambda}f,Sg\rangle
\\&=\langle S^* M_{m,\Lambda,\Lambda}f,g\rangle.
\end{align*}
Therefore $M_{m,\Lambda,\Lambda S}= M_{m,\Lambda,\Lambda}S$ and $M_{m,\Lambda S,\Lambda}=S^* M_{m,\Lambda,\Lambda}.$
If (i) is satisfied, then by \cite[Proposition 3.3]{alizad}, $M_{m,\Lambda,\Lambda}\in GL(\h),$ and
if (ii) is satisfied, then by the Proposition \ref{cooor}, $M_{m,\Lambda,\Lambda}\in GL(\h),$ and so the proof is completed.
\end{proof}

\begin{cor}
Let $\Lambda=\{\Lambda_{\omega}\in B(\h,\K_{\omega}):\omega \in \Omega\}$ be a continuous $g$-frame. Suppose $m\in L^{\infty}(\Omega,\mu)$ is satisfied in one of the following conditions:
\begin{enumerate}
\item[(i)] for some positive constant $\delta,$ $m(\omega)\geq\delta>0$ a.e.
\item[(ii)] $\|m-1\|_\infty\leq\lambda<\frac{A_\Lambda}{B_\Lambda}$ for some $\lambda.$
\end{enumerate}
Then the operators $M_{m,\Lambda,\widetilde{\Lambda}}$ and $M_{m,\widetilde{\Lambda},\Lambda}$ are invertible and 
\begin{eqnarray*}
M_{m,\Lambda,\widetilde{\Lambda}}^{-1}=S_\Lambda M_{m,\Lambda,\Lambda}^{-1},\quad\quad M_{m,\widetilde{\Lambda},\Lambda}^{-1}=M_{m,\Lambda,\Lambda}^{-1}S_\Lambda.
\end{eqnarray*}
\end{cor}
\begin{proof}
 By the Proposition \ref{newn}, for $S=S_\Lambda^{-1}$ the proof is completed.
\end{proof}

\begin{thm}
Let $\Lambda=\{\Lambda_{\omega}\in B(\h,\K_{\omega}):\omega \in \Omega\}$ and $\Lambda^d=\{\Lambda_{\omega}^d\in B(\h,\K_{\omega}):\omega \in \Omega\}$ be dual continuous $g$-frames. 
Let $m\in L^{\infty}(\Omega,\mu)$ such that $\|m-1\|_\infty\leq\lambda<\frac{1}{\sqrt{B_\Lambda B_{\Lambda^d}}}$ for some $\lambda.$
Then $M_{m,\Lambda,\Lambda^d}\in GL(\h)$ and
\begin{align}\label{pers}
\frac{1}{1+\lambda\sqrt{B_\Lambda B_{\Lambda^d}}}\|f\|\leq\|M_{m,\Lambda,\Lambda^d}^{-1}f\|\leq\frac{1}{1-\lambda\sqrt{B_\Lambda B_{\Lambda^d}}}\|f\|,\quad f\in\h,
\end{align}
and 
\begin{align}\label{terac}
M_{m,\Lambda,\Lambda^d}^{-1}=\sum_{k=0}^{\infty}\big(M_{(1-m),\Lambda,\Lambda^d})^k.
\end{align}
where
\begin{align*}
&\big\|M_{m,\Lambda,\Lambda^d}^{-1}-\sum_{k=0}^{n}\big(M_{(1-m),\Lambda,\Lambda^d})^k\big\|
\leq\frac{(\lambda\sqrt{B_\Lambda B_{\Lambda^d}})^{n+1}}{1-\lambda\sqrt{B_\Lambda B_{\Lambda^d}}},\quad n\in\mathbb{N}.
\end{align*}
\end{thm}

\begin{proof}
For any $f,g\in\h$
\begin{align*}
\big|\langle M_{m,\Lambda,\Lambda^d}f-f,g\rangle\big|
&=\big|\langle M_{m,\Lambda,\Lambda^d}f,g\rangle-\langle f,g\rangle\big|
\\&=\Big|\int_\Omega(m(\omega)-1)\langle\Lambda_\omega^d f,\Lambda_\omega g\rangle d\mu(\omega)\Big|
\\&\leq||m(\omega)-1||_\infty\int_\Omega||\Lambda_\omega^d f||||\Lambda_\omega g|| d\mu(\omega)
\\&\leq||m(\omega)-1||_\infty\Big(\int_\Omega||\Lambda_\omega^d f||^2 d\mu(\omega)\Big)^{\frac{1}{2}}\Big(\int_\Omega||\Lambda_\omega g||^2d\mu(\omega)\Big)^{\frac{1}{2}}
\\&\leq\lambda\sqrt{B_\Lambda B_{\Lambda^d}}\|f\|\|g\|.
\end{align*}
Therefore 
\begin{align*}
\|M_{m,\Lambda,\Lambda^d}f-f\|
\leq\lambda\sqrt{B_\Lambda B_{\Lambda^d}}\|f\|.
\end{align*}
Since $\lambda\sqrt{B_\Lambda B_{\Lambda^d}}<1=\frac{1}{\|I^{-1}\|}$ and $I-M_{m,\Lambda,\Lambda^d}=M_{(1-m),\Lambda,\Lambda^d},$ by \cite[Proposition 2.2.]{dstova}, the inequlity (\ref{pers}) and the equlity (\ref{terac}) are satisfied. Also for $n\in\mathbb{N}$ we have
\begin{align*}
\big\|M_{m,\Lambda,\Lambda^d}^{-1}-\sum_{k=0}^{n}\big(M_{(1-m),\Lambda,\Lambda^d})^k\big\|
&=\big\|\sum_{k=n+1}^{\infty}\big(M_{(1-m),\Lambda,\Lambda^d})^k\big\|
\\&\leq\sum_{k=n+1}^{\infty}\|M_{(1-m),\Lambda,\Lambda^d}\|^k
\\&\leq\sum_{k=n+1}^{\infty}(\lambda\sqrt{B_\Lambda B_{\Lambda^d}})^k
\\&=\frac{(\lambda\sqrt{B_\Lambda B_{\Lambda^d}})^{n+1}}{1-\lambda\sqrt{B_\Lambda B_{\Lambda^d}}}.
\end{align*}
\end{proof}
%%%%%%%%%%%%%%%%%%%%%%%%%%%%%%%%%%%%%%%%%%%%%%%%%%%%%%%%
\begin{prop}\label{dfdf}
Let $\Lambda=\{\Lambda_{\omega}\in B(\h,\K_{\omega}):\omega \in \Omega\}$ be a continuous $g$-frame and $\Theta=\{\Theta_\omega \in B(\h,\K_\omega):\omega \in \Omega \}$ be a family of operators such that for each $f\in\h,$ $\{\Theta_\omega f\}_{\omega\in\Omega}$ is strongly measurable that the inequlity (\ref{ffirs}) is satisfied
%\begin{eqnarray}\label{dddd}
%\int_\Omega\|(\Lambda_\omega-\Theta_\omega) f\|^2 d\mu(\omega)\leq\nu\|f\|^2,\quad f\in\h,
%\end{eqnarray}
for some $\nu>0.$ If $\nu<A_\Lambda,$ then $\Theta$ is a continuous $g$-frame. 
\end{prop}
\begin{proof}
By inequlity (\ref{ffirs}), the family $\Lambda-\Theta=\{\Lambda_\omega-\Theta_\omega\in B(\h,\K_{\omega}):\omega \in \Omega\}$ is a continuous $g$-Bessel family and so $\Theta$ is a continuous $g$-Bessel family. For any $f,g\in\h$ we have
\begin{align*}
|\langle M_{1,\widetilde{\Lambda},\Theta}f-f,g\rangle |
&=|\langle M_{1,\widetilde{\Lambda},\Theta}f-M_{1,\widetilde{\Lambda},\Lambda} f,g\rangle |
\\&=\Big|\int_\Omega\langle(\Theta_\omega-\Lambda_\omega) f,\widetilde{\Lambda}_\omega g\rangle d\mu(\omega)\Big|
\\&\leq\Big(\int_\Omega||(\Theta_\omega-\Lambda_\omega) f||^2d\mu(\omega)\Big)^{\frac{1}{2}}\Big(\int_\Omega||\widetilde{\Lambda}_\omega g||^2d\mu(\omega)\Big)^{\frac{1}{2}}
\\&\leq\sqrt{\nu \frac{1}{A_\Lambda}}\|f\|\|g\|.
\end{align*}
Thus
\begin{align*}
\|I-M_{1,\widetilde{\Lambda},\Theta}\|\leq\sqrt{\nu \frac{1}{A_\Lambda}}<1.
\end{align*}
It shows that $M_{1,\widetilde{\Lambda},\Theta}\in GL(\h)$ and then by the Proposition \ref{cong} (ii), $\Theta$ is a continuous $g$-frame.
\end{proof}
$$\textbf{Acknowledgment:}$$

\end{document}